\documentclass[12pt,english]{amsart}

\usepackage{amsmath}
\usepackage{amsfonts}
\usepackage{mathrsfs}
\usepackage{amsthm}
\usepackage{amsbsy}
\usepackage{amssymb}
\usepackage[margin=1in]{geometry}
\usepackage{graphicx}
\usepackage{enumitem} 

\usepackage{hyperref}
\usepackage[alphabetic]{amsrefs} 
\usepackage{color}
\definecolor{atomictangerine}{rgb}{1.0, 0.6, 0.4}

\usepackage[dvipsnames]{xcolor}
\newtheorem{theorem}{Theorem}

\newtheorem{lemma}{Lemma}[section]
\newtheorem{corollary}{Corollary}

\newtheorem*{corollary*}{Corollary}
\newtheorem{proposition}[lemma]{Proposition}

\newtheorem*{conj}{Conjecture}

\theoremstyle{definition}

\newtheorem{notation}[lemma]{Notation}
\newtheorem{remark}[lemma]{Remark}

\newtheorem{definition}[lemma]{Definition}

\newcommand\QQ{{\mathbb Q}}
\newcommand\PP{{\mathbb P}}

\newcommand\p{{\mathbb P}}

\newcommand\Per{{\mathrm{Per}}}
\newcommand\Sc{{\mathbb S}}

\newcommand\tr{\hbox to 1mm  {${}^t \!  $} }

\newcommand{\nc}{\newcommand}

\DeclareMathOperator{\Tail}{Tail}
\DeclareMathOperator{\PrePer}{PrePer}

\nc{\apl}{K\cup \{\infty\}}
\nc{\mfp}{\mathfrak{p}}
\nc{\dmfp}{\delta_\mfp}
\nc{\vmfp}{v_\mfp}
\nc{\ace}{\`e }
\nc{\aca}{\`a }
\nc{\aci}{\`i }
\nc{\aco}{\`o }
\nc{\acu}{\`u }
\nc{\pid}{\mathfrak{p} }
\nc{\bdm}{\begin{displaymath}}
\nc{\edm}{\end{displaymath}}
\nc{\beq}{\begin{equation}}
\nc{\eeq}{\end{equation}}
\nc{\dpid}{\delta_{\mathfrak{p}}}
\nc{\vvs}{\textquotedblleft}
\nc{\vvd}{\textquotedblright}
\nc{\os}{\mathcal{O_S}}

\nc{\mcu}{\mathcal{U}}
\nc{\rs}{\sqrt{R_S}}
\nc{\rsu}{\sqrt{R_S^*}}
\nc{\bpm}{\begin{pmatrix}}
\nc{\epm}{\end{pmatrix}}
\nc{\srs}{\sqrt{R_S}}
\nc{\un}{\{1,2,\ldots,n\}}

\numberwithin{equation}{section}

\title[Scarcity of orbits]{Scarcity of finite orbits for rational functions over a number field}

\author{Jung Kyu Canci}
\address{Jung Kyu Canci, Universit\"{a}t Basel, Mathematisches Institut, Spiegelgasse $1$, CH-$4051$ Basel}
\email{jungkyu.canci@unibas.ch}
\author{Sebastian Troncoso}
\address{Sebastian Troncoso, Birmingham-Southern College, Box 549032,
900 Arkadelphia Road, Birmingham, Alabama 35254, USA}
\email{troncosomath@gmail.com}
\author{Solomon Vishkautsan}
\address{Solomon Vishkautsan, Universit\"{a}tsstra$\beta$e 30, 95447 Bayreuth, Germany}
\email{wishcow@gmail.com}

\begin{document}

\begin{abstract}
Let $\phi$ be a an endomorphism of degree $d\geq{2}$ of the projective line, defined over a number field $K$. Let $S$ be a finite set of places of $K$, including the archimedean places, such that $\phi$ has good reduction outside of $S$. The article presents two main results: the first result is a bound on the number of $K$-rational preperiodic points of $\phi$ in terms of the cardinality of the set $S$ and the degree $d$ of the endomorphism $\phi$. This bound is quadratic in terms of $d$ which is a significant improvement to all previous bounds on the number of preperiodic points in terms of the degree $d$. For the second result, if we assume that there is a $K$-rational periodic point of period at least two, then there exists a bound on the number of $K$-rational preperiodic points of $\phi$ that is linear in terms of the degree $d$.
\end{abstract}

\maketitle

\section{Introduction}  
In this article we prove the following theorem

\begin{theorem}\label{Thm:NFPrePer}
Let $K$ be a number field and $S$ a finite set of places of $K$ containing all the archimedean ones. 
Let $\phi$ be an endomorphism of $\PP^1$, defined over $K$, and $d\geq 2$ the degree of $\phi$. Assume $\phi$ has good reduction outside $S$. Then the number of $K$-rational preperiodic points is bounded by
$$Q(|S|,d)=\alpha_1 d^2+\beta_1 d+\gamma_1,$$
where $\alpha_1$, $\beta_1$ and $\gamma_1$ are positive integers depending only on the cardinality of $S$ and can be effectively computed. 

In addition, if we assume that $\phi$ has a $K$-rational periodic point of period at least two then
the number of $K$-rational preperiodic points is bounded by 
$$L(|S|,d)=\alpha_2 d+\beta_2,$$
where $\alpha_2$ and $\beta_2$ are positive integers depending only on the cardinality of $S$ and can be effectively computed. 
\end{theorem}

We emphasize that the constants $\alpha_1,\alpha_2,\beta_1, \beta_2$ and $\gamma_1$ in the theorem depend only on the cardinality of $S$ (and thus implicitly on the degree $[K:\mathbb{Q}]$) but not on the field $K$ itself. An explicit definition of the bounds $Q(|S|,d)$ and $L(|S|,d)$ will be given in at the end Section \ref{Appendix}. 

Let $\phi:\mathbb{P}_N\to\mathbb{P}_N$ be a endomorphism of degree $d\geq{2}$ defined over a number field $K$. Let $\phi^n$ denote the $n^{th}$ iterate of $\phi$ under composition and $\phi^0$ the identity map. The \emph{orbit} of  $P\in \PP_N(K)$ under $\phi$ is the set  $O_\phi(P)=\{\phi^n(P )  :  n \geq 0 \}$. A point $P \in \PP_N(K)$ is called \emph{periodic} under $\phi$ if there is an integer $n > 0$ such that $\phi^n(P)=P$; the minimal such $n$ is called the \emph{period} of $P$. It is called \emph{preperiodic} under $\phi$ if there is an integer $m \geq 0$ such that $\phi^m(P)$ is periodic. A point that is preperiodic but not periodic is called a \emph{tail} point. Let $\Tail(\phi,K)$, $\Per(\phi,K)$ and $\PrePer(\phi,K)$ denote the sets of $K$-rational tail, periodic and preperiodic points of $\phi$, respectively.

The set of preperiodic points in $\PP_N(\bar{K})$ of an endomorphism  $\phi : \PP_N \rightarrow \PP_N$ of degree $d\geq 2$ defined over a number field $K$, where $\bar{K}$ is the algebraic closure of $K$, is of bounded height (this is a special case of Northcott's theorem \cite{N1950}). Since a number field $K$ possesses the Northcott property (\emph{i.e.}, that every set of bounded height is finite, \cite{BG2006}), the set of $K$-rational preperiodic points of $\phi$ is finite. In fact, from Northcott's proof, an explicit bound on $\PrePer(\phi,K)$ can be found in terms of the coefficients of $\phi$. The problem is to find a bound on the number of preperiodic points that depends in a ``minimal'' way on the map $\phi$. One of the main motivations for our research is the well known Morton and Silverman \cite{MS1994} conjecture which we state below

\begin{conj} [Uniform Boundedness Conjecture] 
Let $\phi:\mathbb{P}_N\to\mathbb{P}_N$ be an endomorphism of degree $d\geq{2}$ defined over a number field $K$. Let $D$ be the degree of $K$ over $\QQ$. Then there exists a number $C=C(D,N,d)$ such that $\phi$ has at most $C$ preperiodic points in $\PP^N(K)$.
\end{conj}


This conjecture is an extremely strong uniformity conjecture. For example, the UBC (Uniform Boundedness Conjecture) for maps of degree 4 on $\PP_1$ defined over a number field $K$ implies Merel's theorem (see \cite{M1996}), i.e. that the size of the torsion subgroup of an elliptic curve over a number field $K$ is bounded only in terms of $[K:\mathbb{Q}]$. The conjecture can also be similarly applied to uniform boundedness of torsion subgroups of abelian varieties; for more details see \cite{F2001}.	

The Uniform Boundedness Conjecture seems extremely difficult to prove even in the simplest case when $(K,N,d)=(\mathbb{Q},1,2)$. Furthermore, in this special case, explicit conjectures have been formulated. For instance, B.\ Poonen \cite{P1998} conjectured an explicit bound when $\phi$ is a quadratic polynomial map over $\mathbb{Q}$. Since every such quadratic polynomial map is conjugate to a polynomial of the form $\phi_c(z)=z^2+c$ with $c\in \mathbb{Q}$ we can state Poonen's conjecture as follows: Let $\phi_c \in \mathbb{Q}[z]$ be a polynomial of degree 2 of the form $\phi_c(z)=z^2+c$ with $c\in \mathbb{Q}$. Then $ | \PrePer(\phi_c,\mathbb{Q}) | \leq 9$. B. Hutz and P. Ingram \cite{HI2013} have shown that Poonen's conjecture holds when the numerator and denominator of $c$  don't exceed  $10^8$.

A natural relaxation of the uniform boundedness conjecture is to restrict our study to families of rational functions given in terms of good reduction. A rational map $\phi:\mathbb{P}_1\to\mathbb{P}_1$ of degree $d\geq{2}$ defined over a number field $K$ is said to have \emph{good reduction} at a non zero prime $\mfp$ of $K$ if $\phi$ can be written as $\phi=[F(X,Y):G(X,Y)]$ where $F,G\in R_\mfp[X,Y]$ are homogeneous polynomials of degree $d$, such that the resultant of $F$ and $G$ is a $\mfp$-unit, where $R_\mfp$ is the localization of the ring of integers of $K$ at $\mfp$. The map $\phi$ is said to have \emph{bad reduction} at a prime $\mfp$ of $K$ if $\phi$ does not have good reduction at $\mfp$. For a fixed finite set $S$ of places of $K$ containing all the archimedean ones, we say that $\phi$ has good reduction outside of $S$ if it has good reduction at each place $\mfp \notin S$.

In the special case of rational functions $\phi:\mathbb{P}_1\to\mathbb{P}_1$, there are several results giving a uniform bound on the number of periodic/preperiodic points of $\phi$ depending on the cardinality of a finite set of places $S$, which includes all archimedean places, together with the constants $[K:\QQ]$ and $\deg(\phi)$, under the assumption that $\phi$ has good reduction outside of $S$ (e.g., \cites{N1989,MS1994,B2007,C2007,C2010,CP2014,CV2016,T2017 }).

We recall the definition of the $\mfp$-adic logarithmic distance on $\p_1(K)$ for a finite place $\pid$ of a number field $K$:
Let $P_1=[x_1:y_1]$ and $P_2=[x_2:y_2]$ be points in $\p_1(K)$. We will denote by 
$$ \dmfp(P_1,P_2)=v_\pid(x_1y_2-x_2y_1)-\min\{v_{\pid}(x_1),v_{\pid}(y_1) \} -\min\{v_{\pid}(x_2),v_{\pid}(y_2) \} $$
the \emph{$\mfp$-adic logarithmic distance} between the points $P_1$ and $P_2$. 

In \cite{CV2016} the first and third author proved a bound on the number of periodic points of a rational function $\phi$ that is \emph{linear} in the degree of $\phi$, but exponential in $|S|$. Roughly, the number of periodic points is bounded by $2^{2^5|S|}d+2^{2^{77}|S|}$. To prove this result, the authors used the following lemma.


\begin{lemma}[Four-Point Lemma A \cite{CV2016}]\label{lem:4P}   
Let $\phi$ be an endomorphism of $\PP_1$ of degree $d\geq{2}$, defined over a number field $K$. Let $S$ be a finite set of places of $K$ containing all the archimedean ones, such that $\phi$ has good reduction outside $S$. Let $A,B,C,D\in\p_1(K)$ be four distinct points such that also the images $\phi(A),\phi(B),\phi(C),\phi(D)$
are distinct. Let $\mathcal{P}$ be the set of points
$P\in \p_1(K)$ satisfying the following four equations for all $\pid\notin{S}$.
\begin{align}\label{main-lemma-eq}
\begin{split}
\dpid(A,P)&=\dpid(\phi(A),\phi(P)),\quad \dpid(B,P)=\dpid(\phi(B),\phi(P)), \\
\dpid(C,P)&=\dpid(\phi(C),\phi(P)),\quad \dpid(D,P)=\dpid(\phi(D),\phi(P)).
\end{split}
\end{align} 
Then $\mathcal{P}$ is finite and 
\beq|\mathcal{P}|\leq 2^{2^5|S|}d+2^{2^{77}|S|}.\eeq
\end{lemma}

One obtains a bound on the number of periodic points by proving that the set of periodic points is contained in the set $\mathcal{P}$ in Four-Point Lemma A. The lemma is directly related to the Siegel-Mahler theorem (cf.\ \cite{FZ2014}) and Evertse's explicit bound on the number of solutions of the $S$-unit equation (see \cite{E1984}), that can be combined and presented in the following way.

\begin{lemma}[Three-Point Lemma A]\label{Sint3p}
Let $K$ be a number field and $S$ a finite set of places of $K$ containing all the archimedean ones. Let $A,B,
C\in\p_1(K)$ be three distinct points. Let $\mathcal{P}$ be the set of points
$P\in \p_1(K)$ satisfying the following three equations for all $\pid\notin{S}$.
\begin{equation}
\dpid(A,P) = 0, \qquad \dpid(B,P) = 0, \qquad \dpid(C,P) = 0
\end{equation}
Then $\mathcal{P}$ is finite and 
\begin{equation}
  |\mathcal{P}|\leq 3\cdot 7^{4|S|}.
\end{equation}
\end{lemma}

The second author proved in \cite{T2017} an inverse relationship between periodic points and tail points. Using Three-Point Lemma A, and combining with results from \cite{CV2016}, we can improve the result in \cite{T2017}. Before stating the result we recall some definitions: A point $[a:b]\in\p_1(K)$ is \emph{critical} for a rational function $\phi\colon \p_1\to\p_1$ if the order of zero at $[a:b]$ of the algebraic condition $\phi([x:y])=\phi([a:b])$ is greater than $1$ (if the point $P$ and its image $\phi(P)$ are finite, then $P$ is critical if and only $\phi'(P)=0$). A point \emph{belongs to a critical cycle} of $\phi$ if it belongs to the orbit of a critical periodic point. We denote by $\textrm{Per}_0(\phi,K)$ the set of $K$-rational (periodic) points that belong to some critical cycle of $\phi$.


\begin{theorem}\label{Thm:RefineTroncoso}
  Let $\phi : \PP_1\to\PP_1$ be a rational map of degree $d\geq{2}$ defined over a number field $K$. 
Suppose $\phi$ has good reduction outside a finite set of places $S$, including all archimedean ones.
\begin{enumerate}
\item If there exist three $K$-rational points such that each point is either a tail point or belongs to a critical cycle of $\phi$ then \[|\Per(\phi,K)| \leq 3\cdot 7^{4|S|} + 3.\]
\item If there exist at least four $K$-rational periodic points of $\phi$ then 
\[|\Tail(\phi,K)| + |\Per_0(\phi,K)| \leq 12\cdot 7^{4|S|}.\]
\end{enumerate}
 
\end{theorem}

It is important to remark that the bounds in the theorem are \emph{independent} of the degree $d$ of the map $\phi$. Furthermore, the hypothesis of the second statement in the result is optimal in the sense that it is impossible to get any result of this form (i.e. independent of $d$) when $|\Per(\phi,K)|<4$, as shown by the examples in \cite[\S{5}]{T2017}.

The main goal of this article is to combine the results from \cite{CV2016} and \cite{T2017} as well as techniques from Canci and Paladino \cite{CP2014}, to obtain a bound on the number of preperiodic points which is better than previous results in terms of the degree $d$ of $\phi$. In fact if the rational map has a $K$-rational periodic point of period at least 2 then we prove a bound on the number of preperiodic points of a rational function that is \emph{linear} in the degree of the rational function, but exponential in $|S|$. In any other case, we provide a bound on the number of preperiodic points of a rational function that is \emph{quadratic} in the degree of the rational function, but exponential in $|S|$.

Note that Benedetto's bound \cite{B2007} on the number of preperiodic points of polynomial maps is $O(d^2/\log{d})$ (for \emph{most} cases of polynomial maps) has been the best estimate so far (and only for polynomials!). 

In the process of proving this theorem, we prove and use the following ``n-point lemmas'', which are generalizations of Three-Point Lemma A. The benefit of these lemmas over Four-Point Lemma A above is that they are independent of the map $\phi$, and therefore much more useful.

 
\begin{lemma}[Three-Point Lemma B] Let $K$ be a number field and $S$ a finite set of places of $K$ containing all the archimedean ones. Let $Q_1,Q_2,Q_3$ be three distinct points in $\p_1(K)$. Given a fixed choice of nonnegative integers $n_{i,\pid}$ for each $i\in\{1,2,3\}$ and each $\pid\notin S$, the set 
\beq\{P\in\p_1(K)\mid \dpid(P,Q_i)=n_{i,\pid}\ \text{for each $i\in\{1,2,3\}$ and $\pid\notin S$}\}\eeq
 has cardinality bounded by a number $B(|S|)$, depending only on $|S|$ \emph{(see Section~\ref{Appendix} for an explicit choice of $B(|S|)$)}.
\end{lemma}

\begin{lemma}[Four-Point Lemma B]
Let $K$ be a number field and $S$ a finite set of places of $K$ containing all the archimedean ones. Let $Q_1,Q_2,Q_3, Q_4$ be four distinct points in $\p_1(K)$. The set 
$$\mathcal{F}=\left\{P\in \p_1(K)\mid \dmfp(P,Q_1)=\dmfp(P,Q_2), \dmfp(P,Q_3)=\dmfp(P,Q_4), \forall\mfp\notin S\right\}$$
is finite with cardinality bounded by a number $C(3,|S|)+2,$ depending only on $|S|$ \emph{(see Section~\ref{Appendix} for an explicit choice of $C(3,|S|)$)}. 
\end{lemma}

\begin{lemma}[Three-Point Lemma C]
Let $K$ be a number field and $S$ a finite set of places of $K$ containing all the archimedean ones. Let $Q_1,Q_2,Q_3$ be three distinct points in $\p_1(K)$. The set 
$$\mathcal{T}=\left\{P\in \p_1(K)\mid \dmfp(P,Q_1)=\dmfp(P,Q_2)=\dmfp(P,Q_3), \mfp\notin S\right\}$$
is finite with cardinality bounded by a number $B(|S|)$ \emph{(same bound as in Three-Point Lemma B)}. 
\end{lemma}


\section{Preliminaries}
%
%
%
In the present article we will use the following notation:

\begin{notation}
$K$ a number field;

$\bar{K}$ an algebraic closure of $K$;

$R$ the ring of integers of $K$;

$\pid$ a non-zero prime ideal of $R$;

$v_{\pid}$ the $\pid$-adic valuation on $K$ corresponding to the prime ideal $\pid$ (we always assume

$v_{\pid}$ to be normalized so that  $v_{\pid}(K^*)=\mathbb{Z}$);

$S$ a fixed finite set of places of $K$ including all archimedean
places;

$R_S = \{x \in K : v _\pid (x) \geq 0 \mbox{    for every prime ideal     } \pid \notin S\}$ the ring of $S$-integers;

$R^{*}_S = \{x \in K : v _\pid (x) = 0 \mbox{    for every prime ideal     } \pid \notin S\}$ the group of $S$-units;

$\Per(\phi,K)$ the set of $K$-rational periodic points of $\phi$;

$\Tail(\phi,K)$ the set of $K$-rational tail points of $\phi$;

$\PrePer(\phi,K)$ the set of $K$-rational preperiodic points of $\phi$;

${\rm Tail}(\phi,K, P)$ the set of $K$-rational tail points of $P$ with respect to $\phi$.

$\Per_0(\phi,K)$ the set of $K$-rational periodic points belonging to some critical cycle of $\phi$.

\end{notation}

We start this section by recalling the definition of the $\mfp$-adic logarithmic distance on $\p_1(K)$ for a finite place $\pid$ in a number field $K$.

\begin{definition}
Let $P_1=[x_1:y_1]$ and $P_2=[x_2:y_2]$ be points in $\p_1(K)$. We will denote by 
$$ \dmfp(P_1,P_2)=v_\pid(x_1y_2-x_2y_1)-\min\{v_{\pid}(x_1),v_{\pid}(y_1) \} -\min\{v_{\pid}(x_2),v_{\pid}(y_2) \} $$
the $\mfp$-adic logarithmic distance on $\p_1(K)$ between the points $P_1$ and $P_2$.
\end{definition}

Note that $\delta_\pid (P_1,P_2)$ is independent of the choice of homogeneous coordinates and $\dmfp(P_1,P_2)$ is $0$ if and only if the points $P_1$ and $P_2$ are distinct modulo $\pid$. 

%

The following definition introduces the idea of normalized forms with respect to $\pid$.

\begin{definition} \quad\\
\begin{enumerate}
\item We say that $P =[x:y] \in\p_1(K)$ is in normalized form with respect to $\pid$ if
$$\min\{ v_\pid(x), v_\pid(y)\} = 0 .$$

\item Let $\phi$ be an endomorphism of $\p_1$, defined over $K$. Assume $\phi$ is given by
$$ \phi = [F(X,Y):G(X,Y)]$$
where $F,G \in K[X,Y]$ are homogeneous polynomials with no common factors. We say that the pair $(F,G)$ is normalized with respect to $\pid$ or that $\phi$ is in normalized form with respect to $\pid$ if $F,G\in R_\pid[X,Y]$ and at least one coefficient of $F$ or $G$ is not in the maximal ideal of $R_\pid$. Equivalently, $\phi=[F:G]$ is normalized with respect to $\pid$ if
$$F(X,Y)=a_0X^d+a_1X^{d-1}Y+...+a_{d-1}XY^{d-1}+a_dY^d $$
and
$$G(X,Y)=b_0X^d+b_1X^{d-1}Y+...+b_{d-1}XY^{d-1}+b_dY^d $$
satisfy
$$\min\{v_\pid(a_0),..., v_\pid(a_d), v_\pid(b_0), ....., v_\pid(b_d)\} =0.$$

\end{enumerate}
\end{definition}

\begin{remark} \label{Remark2.5}
Note that if $P=[x_1:x_2]$ and $Q=[y_1:y_2]$ are in normalized form with respect to $\pid$ then $\delta_\pid(P_1,P_2)=v_{\pid}(x_1y_2-x_2y_1).$
\end{remark}

Since $R_\pid$ is a discrete valuation ring, we can always find a representation of $P$ and $\phi$ in normalized form with respect to $\pid$. However, it is not always true that the same representation is normalized for every $\pid$.

\begin{definition}
Consider $P\in\p_1(K)$ and write $P=[a:b]$ with $a,b\in R_S$. We say that $[a:b]$ are \emph{$S$-coprime coordinates} for $P$ if $\min\{v_\pid(a),v_\pid(b) \}$=0 for every prime $\pid\notin S$.
\end{definition}

Even though the definition of good reduction was given in the introduction, we recall the definition below for convenience of the reader.

\begin{definition}
Let $\phi $ be an endomorphism of $\p_1$, defined over $K$ and write $\phi=[F:G]$ in normalized form with respect to $\pid$. We say that $\phi$ has \emph{good reduction} at $\pid$ if $\tilde{F}(X,Y)=\tilde{G}(X,Y)=0$ has no solutions in $\p_1(\bar{k})$, where $\tilde{F}$ and $\tilde{G}$ are the reductions of $F$ and $G$ modulo $\pid$ respectively and $k$ is the residue field of $R_\pid$. We say that $\phi$ has \emph{good reduction outside $S$} if $\phi$ has good reduction at $\pid$ for every $\pid \notin S$.
\end{definition}

Let $\phi$ be an endomorphism of $\PP^1$, defined over $K$ and $P \in \p_1 (K)$. We say that a point $Q\in \PP^1(K)$ is in the \emph{tail of} $P$ if $P$ is in the orbit of $Q$ under $\phi$. We denote ${\rm Tail}(\phi,K, P)$ the set of $K$-rational tail points of $P$ with respect to $\phi$. Notice that a point $R\in \p_1(K)$ is a tail point if it is non-periodic and in the tail of some periodic point. We define the \emph{tail length} of a tail point $R$ as the minimal natural number $n$ such that $\phi^n(P)$ is periodic. 

Between the many interesting properties of the $\pid$-adic logarithmic distance we state the following that will be relevant for our work.

\begin{proposition}\label{5.1}\emph{\cite[Proposition 5.1]{MS1995}} For all $P_1,P_2,P_3\in\mathbb{P}_1(K)$, we have
\beq\label{pro:td}\delta_{\mathfrak{p}}(P_1,P_3)\geq \min\{\delta_{\mathfrak{p}}(P_1,P_2),\delta_{\mathfrak{p}}(P_2,P_3)\}.\eeq
\end{proposition}

\begin{proposition}\label{5.2}\emph{\cite[Proposition 5.2]{MS1995}}
Let $\phi$ be an endomorphism of $\p_1$ defined over $K$ with good reduction at $\pid$. Then for any $P,Q\in\p(K)$ we have
\beq\label{pro:igr}\delta_{\mathfrak{p}}(\phi(P),\phi(Q))\geq \delta_{\mathfrak{p}}(P,Q).\eeq
\end{proposition}

A direct application of the previous two propositions was deduced by Canci and Paladino \cite{CP2014}.

\begin{lemma}\label{pab} \emph{\cite[Lemma 4.1]{CP2014}}
Let $\phi$ be an endomorphism of $\p_1$ defined over $K$ with good reduction at $\pid$. Let $P_0\in\p_1(K)$ be a fixed point of $\phi$. Let $a,b$ be integers with $0<a<  b$ and $P_b,P_a\in\p_1(K)$ such that $\phi^b(P_b)=\phi^a(P_a)=P_0$ and $\phi^{b-a}(P_b)=P_a$. Then
\beq\label{eqdis}\dpid(P_{b},P_{a})=\dpid(P_{b},P_0)\leq \dpid(P_{a},P_0).\eeq
\end{lemma}

\section{Dynamical properties of the logarithmic distance}

As mentioned in the introduction, the second author proved in \cite{T2017} a strong arithmetic relation between $K$-rational tail points and $K$-rational periodic points.  We state this arithmetic relation below:

\begin{proposition} \label{troncoso}  \emph{\cite[Corollary 2.23]{T2017}}
Let $\phi$ be an endomorphism of $\p_1$, defined over $K$. Suppose $\phi$ has good reduction outside $S$. Let $R\in\p_1(K)$ be a tail point and let $n$ be the period of the periodic part of the orbit of $R$. Let $P\in\p_1(K)$ be any periodic point that is not $\phi^{mn}(R)$ for some $m$. Then $\forall \pid \notin S \quad \dpid(P,R)=0$.
\end{proposition}

In a similar vein, the first and third author proved the following proposition relating periodic points and points belonging to a critical cycle.

\begin{proposition} \label{cv-crit}  \emph{\cite[Corollary 2.6]{CV2016}}
Let $\phi$ be an endomorphism of $\p_1$, defined over $K$. Suppose $\phi$ has good reduction outside $S$. Let $P\in\p_1(K)$ be a periodic point and and let $Q\in\p_1(K)$ belong to a critical cycle. Then $\forall \pid \notin S \quad \dpid(P,Q)=0$.
\end{proposition}

We show how the two propositions together with Three-Point Lemma A can be used to prove Theorem~\ref{Thm:RefineTroncoso}.

\begin{proof}[Proof of Theorem~\ref{Thm:RefineTroncoso}] \quad\\
\begin{enumerate}
\item Assume there exist three $K$-rational points $R_1,R_2,R_3$ such each point is either a tail point or belongs to a critical cycle of $\phi$. Then by Propositions~\ref{troncoso} and \ref{cv-crit} we get that for any periodic point $P\in\Per(\phi,K)\setminus\{R_1,R_2,R_3\}$ we have $$\forall \pid \notin S, \quad \forall 1\leq{i}\leq{3}, \quad \dpid(P,R_i)=0.$$ We can therefore apply Three-Point Lemma A to obtain the required bound.
\item Assume there exist four $K$-rational periodic points of $\phi$ which we denote by $P_1$, $P_2$, $P_3$ and $P_4$. Let $R\in{\Tail(\phi,K)\cup\Per_0(\phi,K)}$. If $R$ is a tail point there can exist at most one periodic point $P$ such that $\phi^{mn}(R)=P$ for some $m$, where $n$ is the period of $P$; therefore at least three of $P_1,P_2,P_3,P_4$ do not satisfy this property, and we can assume without loss of generality that $P_1,P_2,P_3$ do not satisfy this property. If $R$ is a periodic point, we can assume again without loss of generality that it is distinct from $P_1,P_2,P_3$. By Propositions~\ref{troncoso} and \ref{cv-crit} we get that $$\forall \pid \notin S, \quad \forall 1\leq{i}\leq{3}, \quad \dpid(P_i,R)=0.$$ We can therefore apply Three-Point Lemma A (four times, depending on which three points of $P_1,P_2,P_3,P_4$ we choose) to obtain the required bound.
\end{enumerate}
\end{proof}

\section{The $n$-point lemmas}

We would like to write every point in $S$-coprime coordinates. However, $R_S$ is generally not a principal ideal domain and thus there exist points in $\p_1(K)$ that do not have $S$-coprime coordinates. To avoid this problem we use the same argument as in \cite{C2007}. For the reader's convenience we write this argument below.

Let  $\mathbf{a_1},...,\mathbf{a_h}$ be a full system of integral representatives for the ideal classes of $R_S$. Hence, for each $i \in \{1,..,h\}$ there is an $S$-integer $\alpha_i \in R_S$ such that
$$ \mathbf{a_i}^h = \alpha_i R_S .$$
Let $L$ be the extension of $K$ given by
$$L=K(\zeta,\sqrt[h] \alpha_1 ,...,\sqrt[h] \alpha_h)$$
where $\zeta$ is a primitive $h$-th root of unity and $\sqrt[h] \alpha_i$ is a chosen $h$-th root of $\alpha_i$.

We denote by $\sqrt{R_S^{*}}$, $\sqrt{R_S}$ and $\sqrt{K}$ the radicals in $L^{*}$ of $R_S^{*}$, $R_S$ and $K$ respectively. Denote by $\mathbb{S}$ the set of places of $L$ lying above the places in $S$ and by $R_\mathbb{S}$ and $R_\mathbb{S}^{*}$ the ring of $\mathbb{S}$-integers and the group of $\mathbb{S}$-units, respectively in $L$. By definition $R_\mathbb{S}^{*} \cap \sqrt{K^{*}}= \sqrt{R_S^{*}}$ and $\sqrt{R_S^{*}}$ is a subgroup of $L^{*}$ of free rank $s-1$ by Dirichlet's unit theorem.

\begin{lemma} \label{repre}
Assume  the notation above. There exist fixed representations $[x_P:y_P] \in \p_1(L)$ for every rational point $P \in \PP^1(K)$ satisfying the following two conditions.
\begin{enumerate}

\item [(a)] For every $P\in\p_1(K)$, we have $x_P,y_P \in \sqrt {K
^{*}}$ and
$$ x_P R_\mathbb{S}+y_P R_\mathbb{S}=R_\mathbb{S}.$$

\item [(b)] If $P,Q\in\PP^1(K)$ then
$$ x_Py_Q-y_Px_Q \in \sqrt{K^{*}} .$$

\end{enumerate}
\end{lemma}

\begin{proof}
Let $P=[x:y]$ be a representation of $P$ in $\p_1(K)$ and consider  $\mathbf{b} \in \{ \mathbf{a_1},...,\mathbf{a_h} \}$ a representative of $x R_S+y R_S$. We can find   $\beta \in K^{*}$ such that $\mathbf{b}^h =\beta R_S$. Then there is $\lambda  \in K^{*}$ such that
\begin{equation} \label{Lemma.R}
(xR_S +yR_S)^h = \lambda^h\beta R_S.
\end{equation}

We define in $L$
$$ x'=\frac{x}{\lambda \sqrt[h]\beta}  \quad\quad\quad y'=\frac{y}{\lambda \sqrt[h]\beta}$$
and with this definition, it is clear that $x',y' \in \sqrt {K^{*}}$ such that $x' R_\mathbb{S}+y' R_\mathbb{S}=R_\mathbb{S}$.
\\ Furthermore, let $P=[x_1',y_1']$ and $Q=[x_2':y_2']$  where
$$ x_i'=\frac{x_i}{\lambda_i \sqrt[h]\beta_i}  \quad\quad\quad y_i'=\frac{y_i}{\lambda_i \sqrt[h]\beta_i}$$
and $\lambda_i ,\beta_i$ are as the ones described in equation (\ref{Lemma.R}) for $i\in \{1,2\}$. Then
$$(x_1'y_2'-y_1'x_2')^h = \frac{(x_1y_2-y_1x_2)^h}{\lambda_1^h \lambda_2^h \beta_1\beta_2} \in K^{*}. $$
\end{proof}

\begin{definition}
A point $P\in\p_1(K)$ written as in Lemma \ref{repre} is said to be written in \emph{$S$--radical coprime coordinates}.
\end{definition} 
 
For the rest of this section we assume the above notation for $L$, $\mathbb{S}$, $\sqrt{K}$, $\sqrt{R_S^{*}}$, and $\sqrt{R_S}$.

%

Before we prove Three-Point Lemma B, Four-Point Lemma B and Three-Point Lemma C, we need to recall some results on the $S$-unit equation.  Below we cite a result from Beukers and Schlickewei that gives a bound on the number of solutions of the $S$-unit equation in two variables where these solutions lie in a multiplicative subgroup.

\begin{theorem} [Beukers and Schlickewei  \cite{BS1996}]
Let $K$ be a number field and $\Gamma$ be a subgroup of $(K^{*})^2=K^{*}\times K^{*}$ of rank $r$. Then the equation
$$x+y=1$$
has at most $2^{8(r+1)}$ solutions with $(x,y) \in \Gamma$.
\end{theorem}

\begin{corollary} \label{S-unit}
Let $K$ be a number field and $\Gamma_0$ be a subgroup of $K^{*}$ of rank $r$. Consider $\Gamma=\Gamma_0 \times \Gamma_0$ and assume $a,b \in K^{*}$. Then the equation
$$ax+by=1 $$
has at most $2^{8(2r+2)}$ solutions with $(x,y) \in \Gamma$.
\end{corollary}

Next we quote a result from Evertse, Schlickewei and Schmidt that gives a bound on the number of solutions of the $S$-unit equation in three or more variables where these solutions lie in a multiplicative subgroup.

\begin{theorem} [Evertse, Schlickewei and Schmidt \cite{ES2002}] \label{S-unit:more variables} 
Let $K$ be a number field and $\Gamma$ be a subgroup of $(K^{*})^n$ of rank $r$. Assume $a_1,\ldots,a_n \in (K^{*})^n$. Then the equation
$$a_1x_1+\ldots+a_nx_n=1$$
has at most $e^{(6n)^{3n}(r+1)}$ solutions with $(x_1,\ldots,x_n) \in \Gamma$ and $\displaystyle \sum_{i\in I} a_ix_i \neq 0$ for every nonempty subset $I$ of $\{1,\ldots,n\}$.
\end{theorem}

In general, we will use Corollary \ref{S-unit} and Theorem \ref{S-unit:more variables} with $\Gamma=\rsu\times \rsu$ and $\Gamma=(\rsu)^n$ respectively. The bounds in each case will be $2^{8(2|S|)}$ and $e^{(6n)^{3n}(n|S|+1-n)}$ and they will be denoted by $B(|S|)$ and $C(n,|S|)$, respectively. In what follows we will prove Three-Point Lemma B, Four-Point Lemma B and Three-Point Lemma C.

\begin{lemma}[Three-Point Lemma B] Let $K$ be a number field and $S$ a finite set of places of $K$ containing all the archimedean ones. Let $Q_1,Q_2,Q_3$ be three distinct points in $\p_1(K)$. Given a fixed choice of nonnegative integers $n_{i,\pid}$ for each $i\in\{1,2,3\}$ and each $\pid\notin S$, the set 
\beq\label{3PfD}\{P\in\p_1(K)\mid \dpid(P,Q_i)=n_{i,\pid}\ \text{for each $i\in\{1,2,3\}$ and $\pid\notin S$}\}\eeq
 has cardinality bounded by $B(|S|)$.
\end{lemma}
\begin{proof}
The set in \eqref{3PfD} is empty if the set of nonzero numbers $n_{i,\pid}$ is infinite. Otherwise there are three elements $C_1,C_2,C_3\in \sqrt{K^{*}}$ 
such that 
$$v_{\pid^{\prime}}(C_i)=n_{i,\pid^{\prime}}$$
for each $i\in\{1,2,3\}$ and every $\pid^{\prime}\notin \Sc$. Assume $Q_i=[a_i:b_i]$ to be written in $S$--radical coprime coordinates for each $i\in\{1,2,3\}$. A point $P=[x:y]$, written in $S$--radical coprime coordinates too, belongs to the set in \eqref{3PfD} if and only if there exist three units $u_1,u_2,u_3\in\rsu$ verifying the three following conditions
$$a_iy-b_ix=u_iC_i$$
for each $i\in\{1,2,3\}$. The three units $u_1,u_2,u_3$ verify the following equation
$$(a_3b_2 - b_3a_2)C_1u_1 + (b_3a_1 - a_3b_1)C_2u_2 = (b_2a_1 - b_1a_2)C_3u_3.$$
Now the proof is a trivial application of Beukers and Schlickewei's result (see Corollary \ref{S-unit}) and algebraic manipulations.
\end{proof}

\begin{proposition}[Four-Point Lemma B]\label{pro:4points}
Let $K$ be a number field and $S$ a finite set of places of $K$ containing all the archimedean ones. Let $Q_1,Q_2,Q_3, Q_4$ be four distinct points in $\p_1(K)$. The set 
$$\mathcal{F}=\left\{P\in \p_1(K)\mid \dmfp(P,Q_1)=\dmfp(P,Q_2), \dmfp(P,Q_3)=\dmfp(P,Q_4), \forall\mfp\notin S\right\}$$
is finite with cardinality bounded by $C(3,|S|)+2$.
\end{proposition}
\begin{proof}
We assume $Q_i=[x_i:y_i]$ written in $S$--radical coprime coordinates for each $i\in\{1,2,3\}$. Without loss of generality we may assume $Q_1=[0:1]$ (up to a $\rs$--invertible change of coordinates of $\p_1$).  Let $P\in \mathcal{F}$ and we assume that $P=[x:y]$ is written in $S$--radical coprime coordinates too. The condition $\dmfp(P,Q_1)=\dmfp(P,Q_2)$ is equivalent to the existence of a unit $u\in\rsu$ such that $x=u(xy_2-yx_2)$, that is equivalent to 
$$y=\frac{x}{ux_2}(uy_2-1).$$
Therefore $P=[ux_2: uy_2-1]$. 
Since 
\[\dmfp(P,Q_3)=v_\mfp(ux_2y_3-(uy_2-1)x_3)-\min\{v_\mfp(ux_2), v_\mfp(uy_2-1)\}\] and 
\[\dmfp(P,Q_4)=v_\mfp(ux_2y_4-(uy_2-1)x_4)-\min\{v_\mfp(ux_2), v_\mfp(uy_2-1)\}\] are equal, there exists a unit $v\in\rsu$ such that 
$$ux_2y_3-(uy_2-1)x_3=v((ux_2)y_4-(uy_2-1)x_4).$$
Therefore the two units $u,v$ have to verify the following equation
\beq\label{pro:eq4p:v2}Au+Bv+Cuv=1\eeq
where $\displaystyle A=\frac{x_2y_3-x_3y_2}{-x_3}$, $\displaystyle B=\frac{x_4}{x_3}$ and $\displaystyle C=\frac{x_4y_2-x_2y_4}{-x_3}$.

We consider all possible vanishing subsums. We see that the cases $Au=0$, $Bv=0$, $Cuv=0$,$Au+Cuv=0$ and $Bv+Cuv=0$ are all impossible because the points $Q_i$ are distinct and the case $Au+Bv=0$ provides two solutions for $u$. Hence $u$ assume at most $C(3,|S|)+2$ possibilities and thus the cardinality of $\mathcal{F}$ is bounded by $C(3,|S|)+2$.

\end{proof}

\begin{proposition}[Three-Point Lemma C]\label{pro:3points}
Let $K$ be a number field and $S$ a finite set of places of $K$ containing all the archimedean ones. Let $Q_1,Q_2,Q_3$ be three distinct points in $\p_1(K)$. The set 
$$\mathcal{T}=\left\{P\in \p_1(K)\mid \dmfp(P,Q_1)=\dmfp(P,Q_2)=\dmfp(P,Q_3), \mfp\notin S\right\}$$
is finite with cardinality bounded by $B(|S|)$. 
\end{proposition}
\begin{proof} This is a particular case of the Three-Point Lemma B proven above.
\end{proof}


\section{Bounds on tail points}\label{Bounds on tail points}

The main objective of this section is to prove Theorem \ref{Thm:NFPrePer}. To do so, we will prove four lemmas to bound the number of tail points of a $K$-rational periodic point. First we provide a bound for $|{\rm Tail}(\phi,K, P)|$ when $P$ is a fixed $K$-rational point under the endomorphism $\phi$.

\begin{lemma}\label{lem:TailFix}
Let $K$ be a number field and $S$ a finite set of places of $K$ containing all the archimedean ones.  
Let $\phi$ be an endomorphism of $\PP^1$, defined over $K$, and $d\geq 2$ the degree of $\phi$. Assume $\phi$ has good reduction outside $S$. Let $P\in\p_1(K)$ be a fixed point of $\phi$. Then 
$$|{\rm Tail}(\phi,K, P)| \leq L_1(d,|S|)$$
where $L_1(d,|S|)=(d-1)(1+d(1+B(|S|))).$
\end{lemma}

\begin{proof}
Consider $P_1$ and $P_2$ in $\p_1(K)$ such that $P_1$ is not periodic, $\phi(P_1)=P$ and $\phi(P_2)=P_1$. If a point such as $P_2$ do not exist then the cardinality of ${\rm Tail}(\phi,K, P)$ is bounded by $d-1$ and thus our bound holds.

Suppose ${\rm Tail}(\phi,K, P_2)$ is non empty and let $Q\in{\rm Tail}(\phi,K, P_2)$. Lemma \ref{pab} implies that 
$$\dmfp(Q,P_2)=\dmfp(Q,P_1)=\dmfp(Q,P) \quad \mbox{for every $\mfp\notin S$}.$$
Therefore it is enough to apply Proposition \ref{pro:3points} to get 
$$|{\rm Tail}(\phi,K, P_2)|\leq B(|S|).$$
Notice that if ${\rm Tail}(\phi,K, P_2)$ is empty then the previous inequality trivially holds. 

Considering the fact that there are $d-1$ possibilities for $P_1$ and $d$ for $P_2$, we get
$$|{\rm Tail}(\phi,K, P)|\leq (d-1)(1+d(1+B(|S|)))$$
as desired.
\end{proof}

Now we provide a bound for $|{\rm Tail}(\phi,K, P)|$ when $P$ is a $K$-rational periodic point of $\phi$ of period 2. 

\begin{lemma}\label{lem:TailTwo}
Let $K$ be a number field and $S$ a finite set of places of $K$ containing all the archimedean ones.  
Let $\phi$ be an endomorphism of $\PP^1$, defined over $K$, and $d\geq 2$ the degree of $\phi$. Assume $\phi$ has good reduction outside $S$. Let $P\in\p_1(K)$ be a periodic point of $\phi$ of period $2$. Then 
$$|{\rm Tail}(\phi,K, P)| \leq L_2(d,|S|)$$
where
\small 
\[ L_2(d,|S|)= \max\{ (2(C(3,|S|) + 2)+1)d + 1, (d-1)(1+B(|S|)(B(|S|) + C(3,|S|) + 2+1)) \}.\]
\normalsize
\end{lemma}


%

\begin{proof} 
Let $\mfp\notin S$. We denote $P_1=P$, and $P_2=\phi(P)$. If ${\rm Tail}(\phi,K, P)$ is empty then the result is trivially true. We assume that ${\rm Tail}(\phi,K, P)$ is not empty and we split the proof into two cases.

\textbf{Case 1: Suppose $P_1$ has a unique non-periodic $K$-rational preimage.} Denote this preimage by $Q\in\p_1(K)$. Consider $R\in\p_1(K)$ a preimage of $Q$ and $T\in\p_1(K)$ a tail point of $R$. If such point $T$ does not exist then the cardinality of ${\rm Tail}(\phi,K, P)$ is bounded by $d+1$ and thus our bound holds.

We note that under application of $\phi^2$, $P_1$ and $P_2$ are fixed, and $R$ becomes a preimage of $P_1$ and $Q$ becomes a preimage of $P_2$; the point $T$ is (for $\phi^2$) either a tail point of $Q$ or a tail point of $R$ depending on the parity of its tail length for $\phi$. 

Without loss of generality (one can see that the situation is symmetric for $\phi^2$), assume that $T$ is a tail point for $Q$ (under $\phi^2$). By Lemma \ref{pab} we get $\dmfp(T, Q) = \dmfp(T,P_2)$ and by Proposition \ref{troncoso} we get $\dmfp(T,P_1)=\dmfp(Q, P_1)=0$ and repeated application of Lemma \ref{5.2} (on $\phi^2$) we get $\dmfp(T,R) \leq \dmfp(Q, P_1) = 0$, which implies $\dmfp(T,R)=0$. \\ Thus $T$ satisfies the following two equations
\[\dmfp(T, Q) = \dmfp(T,P_2) \qquad \dmfp(T,P_1) = \dmfp(T,R) \quad \mbox{for every $\mfp\notin S$}.\]
By Lemma \ref{pro:4points} there are at most $C(3,|S|) + 2$ solutions for $T$. Then 
$$ |{\rm Tail}(\phi^2,K, Q)| \leq C(3,|S|) + 2.$$ 
By symmetry we get
$$ |{\rm Tail}(\phi^2,K, R)| \leq C(3,|S|) + 2.$$ 
Considering that there are $d$ possibilities for $R$ we obtain
\[|{\rm Tail}(\phi,K, P)| \leq (2(C(3,|S|) + 2)+1)d + 1.\]

\textbf{Case 2: Suppose $P_1$ has at least two non-periodic $K$-rational preimages}. There are at most $d-1$ such preimages. Denote two of them by $Q_1$ and $Q_2$ with the property that there is a point $R\in\p_1(K)$ which is a preimage of $Q_1$, if such an $R$ does not exist then $|{\rm Tail}(\phi,K, P)|\leq d-1$.

By \ref{troncoso} we get $\dmfp(R,P_2)=\dmfp(Q_1, P_1)=0$ and by Lemma \ref{5.2} we get that $\dmfp(R,Q_i) \leq \dmfp(Q_1, P_1) = 0$ for $i\in\{1,2\}$. Thus $R$ satisfies
\[ \dmfp(R, Q_1) = \dmfp(R,Q_2)=\dmfp(R,P_2) = 0 \quad \mbox{for every $\mfp\notin S$}.\]

By Lemma \ref{pro:3points} we get there are at most $B(|S|)$ solutions for $R$. Looking again at $\phi^2$, we see that $Q_1$ and $Q_2$ are preimages of the fixed point $P_2$. By similar arguments to the case 1, $R$ has at most $B(|S|) + C(3,|S|) + 2$ tail points (by summing up the cases of whether $T$ is a tail point of $R$ or $Q_1$ under $\phi^2$). Therefore we get  
\[|{\rm Tail}(\phi,K, P)| \leq (d-1)(1+B(|S|)(B(|S|) + C(3,|S|) + 2+1))\]

Considering case 1 and case 2 we obtain 
\[|{\rm Tail}(\phi,K, P)| \leq \max\{ (2(C(3,|S|) + 2)+1)d + 1, (d-1)(1+B(|S|)(B(|S|) + C(3,|S|) + 2+1)) \}.  \]

\end{proof}

The next lemma provides a bound for $|{\rm Tail}(\phi,K, P)|$ when $P$ is a $K$-rational periodic point of $\phi$ of period $3$.

\begin{lemma}\label{lem:TailThree}
Let $K$ be a number field and $S$ a finite set of places of $K$ containing all the archimedean ones.  
Let $\phi$ be an endomorphism of $\PP^1$, defined over $K$, and $d\geq 2$ the degree of $\phi$. Assume $\phi$ has good reduction outside $S$. Let $P\in\p_1(K)$ be a periodic point of $\phi$ of period 3. Then 
$$|{\rm Tail}(\phi,K, P)| \leq L_3(d,|S|),$$
where $L_3(d,|S|)=((1+3B(|S|))B(|S|)+1)(d-1)$.
\end{lemma}

\begin{proof}
Let $\mfp\notin S$. We denote $P_1=P$, $P_2=\phi(P)$, and $P_3=\phi^2(P)$. Consider $Q,R\in\p_1(K)$ such that $Q$ is not periodic, $\phi(Q)=P_1$ and $\phi(R)=Q$. If such point $R$ does not exist then the cardinality of ${\rm Tail}(\phi,K, P)$ is bounded by $d-1$ and thus our bound holds.

By \ref{troncoso} we get $\dmfp(R,P_1)=\dmfp(R,P_3)=\dmfp(Q,P_1)=0$ and by Lemma \ref{5.2} we get that $\dmfp(R,Q) \leq \dmfp(Q, P_1) = 0$. Thus $R$ satisfies
\[ \dmfp(R, P_1) = \dmfp(R,P_3)=\dmfp(R,Q) = 0 \quad \mbox{for every $\mfp\notin S$}.\]
By Lemma \ref{pro:3points} we get there are at most $B(|S|)$ solutions for $R$. Suppose ${\rm Tail}(\phi,K, R)$ is not empty. Let $T\in\p_1(K)$ be a tail point of $R$ and $n$ its tail length for $\phi$.

We note that under application of $\phi^3$, $P_1$, $P_2$ and $P_3$ are fixed, and $R$ becomes a preimage of $P_2$ and $Q$ becomes a preimage of $P_3$; the point $T$ is (for $\phi^3$) either a tail point of $P_1$ or $P_2$ or $P_3$ depending if $n$ is congruent to $1,2$ or $0$ modulo $3$, respectively.

\textbf{Case 1: Suppose $n  \equiv 0 \mbox{ or } 2 \pmod{3}.$}

One can see that the cases $n  \equiv 0 \pmod{3}$ and $n  \equiv  2 \pmod{3}$ are symmetric for $\phi^3$. We assume without loss of generality  that $n  \equiv 0 \pmod{3}$ \emph{i.e.} $T$ is a tail point for $Q$. By Proposition \ref{troncoso} we get $\dmfp(T,P_1)=\dmfp(T, P_2)=\dmfp(Q,P_2)=0$ and repeated application of Lemma \ref{5.2} (on $\phi^3$) we get $\dmfp(T,R) \leq \dmfp(Q, P_2) = 0$, which implies $\dmfp(T,R)=0$. \\ Thus $T$ satisfies the following the equations
\[\dmfp(T, P_1) = \dmfp(T,P_2) = \dmfp(T,R)=0  \quad \mbox{for every $\mfp\notin S$}.\]
By Lemma \ref{pro:3points} there are at most $B(|S|)$ solutions for $T$ for the case $n  \equiv 0 \pmod{3}$, and thus $2B(|S|)$ for both residue classes.


\textbf{Case 2: Suppose $n  \equiv 1 \pmod{3}.$}

Since $n  \equiv 1 \pmod{3}$ we have that $T$ is a tail point for $P_1$. By Proposition \ref{troncoso} we get $\dmfp(T,P_2)=\dmfp(T, P_3)=\dmfp(\phi(T),P_1)=0$ and applying Lemma \ref{5.2} (on $\phi$) we get $\dmfp(T,Q) \leq \dmfp(\phi(T), P_1) = 0$, which implies $\dmfp(T,Q)=0$. \\ Thus $T$ satisfies the following the equations
\[\dmfp(T, P_2) = \dmfp(T,P_3) = \dmfp(T,Q)=0  \quad \mbox{for every $\mfp\notin S$}.\]
By Lemma \ref{pro:3points} there are at most $B(|S|)$ solutions for $T$. 
Considering both cases we have that 
$$ |{\rm Tail}(\phi,K,R )| \leq 3B(|S|).$$ 
Notice that if ${\rm Tail}(\phi,K, R)$ is empty then the previous inequality trivially holds.
Considering that $Q$ has at most $d-1$ $K$-rational preimages and each of those preimages have at most $B(|S|)$ $K$-rational preimages  we obtain
\[|{\rm Tail}(\phi,K, P)| \leq ((1+3B(|S|))B(|S|)+1)(d-1).\]

\end{proof}

The last lemma of this section provides a bound for $|{\rm Tail}(\phi,K, P)|$ when $\phi$ admit a $K$-rational fixed point and a $K$-rational periodic point of period $2$.

\begin{lemma}\label{lem:TailFixedAndTwo}
Let $K$ be a number field and $S$ a finite set of places of $K$ containing all the archimedean ones.  
Let $\phi$ be an endomorphism of $\PP^1$, defined over $K$, and $d\geq 2$ the degree of $\phi$. Assume $\phi$ has good reduction outside $S$. Let $P\in\p_1(K)$ be a fixed point of $\phi$ and $Q\in\p_1(K)$ a periodic point of $\phi$ of period 2. Then 
$$|{\rm Tail}(\phi,K, P)| \leq L_4(d,|S|)$$
where $L_4(d,|S|)=(C(3,|S|) + 2+1)(d -1).$
\end{lemma}

\begin{proof}
Let $\mfp\notin S$. We denote $Q_1=Q$, and $Q_2=\phi(Q)$. Consider $R\in\p_1(K)$  such that $R$ is not periodic and $\phi(R)=P$ and $T\in\p_1(K)$ a tail point of $R$. If such point $T$ does not exist then the cardinality of ${\rm Tail}(\phi,K, P)$ is bounded by $d-1$ and thus our bound holds.

By Lemma \ref{pab} we get $\dmfp(T, R) = \dmfp(T,P_1)$ and by Proposition \ref{troncoso} we get $\dmfp(T,Q_1)=\dmfp(T, Q_2)=0$. Thus $T$ satisfies the following two equations
\[\dmfp(T, R) = \dmfp(T,P_1) \qquad \dmfp(T,Q_1) = \dmfp(T,Q_2) \quad \mbox{for every $\mfp\notin S$}.\]
By Lemma \ref{pro:4points} there are at most $C(3,|S|) + 2$ solutions for $T$. Considering that there are $d-1$ possibilities for $R$ we obtain
\[|{\rm Tail}(\phi,K, P)| \leq (C(3,|S|) + 2+1)(d -1) .\]
\end{proof}

Before we prove Theorem~\ref{Thm:NFPrePer} we emphasize that $L_1(d,|S|)$ is quadratic in terms of $d$ but $L_2(d,|S|)$, $L_3(d,|S|)$, and $L_4(d,|S|)$ are linear in terms of $d$. We will denote by $CV(d,|S|)$ the bound for $|{\rm Per}(\phi,K)|$ mentioned in \cite[Corollary 1]{CV2016} which is linear in terms of $d$ and we denote by $T(|S|)$ the refined result of Troncoso's bound proven in Theorem \ref{Thm:RefineTroncoso} part $(2)$.

Now we have all the tools to prove the main theorem of this paper. 

\begin{proof}[Proof of Theorem  \ref{Thm:NFPrePer}]
If the map $\phi$ has more than three periodic points in $\p_1(K)$, then we have
$$|{\rm PrePer}(\phi,K)|\leq T(|S|)+CV(d,|S|).$$
If the map $\phi$ has at most three periodic points in $\p_1(K)$, we have two cases. First if $\phi$ has a $K$-rational point of period at least $2$ then 
$$|{\rm PrePer}(\phi,K)|\leq \max\{L_4(d,|S|)+2L_2(d,|S|) ,3L_3(d,|S|) \}+3.$$
On the other hand, if $\phi$ has no $K$-rational points of period at least 2 then 
$$|{\rm PrePer}(\phi,K)|\leq 3L_1(d,|S|)+3.$$
\end{proof}

\section{Appendix}\label{Appendix}
For the convenience of the reader in this section we give all the bounds used in this paper in an explicit form. Let $S$ be a finite set of places of a number field $K$ containing all the archimedean ones. Let $n$ be a non negative integer and $d$ a positive integer.

$B(|S|)=2^{16|S|}$;

$C(n,|S|)=e^{(6n)^{3n}(n|S|+1-n)}$;
 
$L_1(d,|S|)=(d-1)(1+d(1+B(|S|)))$;

$L_2(d,|S|)= \max\{ (2(C(3,|S|) + 2)+1)d + 1, (d-1)(1+B(|S|)(B(|S|) + C(3,|S|) + 2+1)) \}$;

$L_3(d,|S|)=((1+3B(|S|))B(|S|)+1)(d-1)$;

$L_4(d,|S|)=(C(3,|S|) + 2+1)(d -1)$;

$CV(d,|S|)=(3B(|S|)+13)d+27B(|S|)+C(5,|S|)+6C(3,|S|)+32$;

$T(|S|)= 12\cdot 7^{4|S|}$;

$L(d,|S|)=\max\{T(|S|)+CV(d,|S|) ,L_4(d,|S|)+2L_2(d,|S|) +3,3L_3(d,|S|)+3 \}$;

$Q(d,|S|)=\max\{T(|S|)+CV(d,|S|) ,3L_1(d,|S|)\}$.



\end{document}